\documentclass[12pt,a4paper]{article}
\usepackage[a4paper, total={6in, 10in}]{geometry}

\usepackage{url}
\usepackage{amsmath}
\usepackage{amsthm}
\usepackage{amsfonts}

\usepackage{authblk}

\usepackage[T1]{fontenc}
\usepackage[utf8]{inputenc}
\newtheorem{definition}{Definition}
\newtheorem{thm}{Theorem}
\newtheorem{lem}{Lemma}

\title{A general explicit form for higher order approximations for fractional derivatives and its consequences}

\author{W. A. Gunarathna\thanks{gunarathnawa@yahoo.com}}
\affil{Department of Mathematics,
General Sir John Kotelawala Defence University, Ratmalana, Sri Lanka}
\author{H. M. Nasir\thanks{nasirh@squ.edu.om}}
\affil{Department of Mathematics,
Sultan Qaboos University, Sultanate of Oman}
\author{W. B. Daundasekera\thanks{wbd@pdn.ac.lk}}
\affil{Department of Mathematics,
University of Peradeniya, Sri Lanka}

\begin{document}

\maketitle

\begin{abstract}
A general explicit form for generating functions for approximating fractional derivatives is derived. To achieve this, an equivalent characterisation for consistency and order of approximations established on a general generating function is used to form a linear system of equations with Vandermonde matrix for the coefficients of the generating function which is in the form of power of a polynomial. This linear system is solved
for the coefficients of the polynomial in the generating function.
These generating functions
completely characterise Grünwald type approximations with shifts and order
of accuracy. Incidentally, the constructed generating functions happen to be generalization of the previously known Lubich forms of generating functions without shift. As a consquence,
a general explicit form for new finite difference formulas for integer-order derivatives with any order of accuracy are derived.
\end{abstract}



\section{Introduction}

Fractional calculus has a history that goes back to L'Hospital,
Leibniz and Euler \cite{leibnitz1962letter, eulero1738progressionibus}.
A historical account of early works on fractional calculus can be found,
for eg., in \cite{ross1977development}.
Fractional integral and fractional derivative are extensions of the integer-order integrals and derivatives to a real or complex order.
Various definitions of fractional derivatives have been proposed in the past, among which
the Riemann-Liouville, Gr\"unwald-Letnikov and Caputo derivatives are common and established.
Each definition characterizes certain properties of the integer-order derivatives.

Recently, fractional calculus found its way into the application domain in
science and engineering \cite{mainardi1996fractional,bagley1983theoretical,vinagre2000some,
metzler2004restaurant}. Fractional derivative is also found to be more suitable to describe
anomalous transport in an external field
derived from the continuous time random walk \cite{barkai2000continuous},
resulting in a fractional diffusion equation (FDE).
The FDE involves fractional derivatives either in time, in space or in both  variables.

Fractional derivative is approximated by
Gr\"unwald approximation obtained from the equivalent Gr\"unwald-Letnikov formula of the Riemann-Liouville definition.
Numerical experience and theoretical justifications have shown that
application of this approximation {\it as it is}
in the space fractional diffusion equation results in unstable solutions when explicit Euler,
implicit Euler and even the Crank-Nicolson (CN) type schemes are used \cite{meerschaert2004finite}.
The latter two schemes are popular for their unconditional stability for
classical diffusion equations with integer-order derivatives.
This peculiar phenomenon for the implicit and CN type schemes is corrected and the
stability is restored when a shifted
form of the Gr\"unwald approximation
is used \cite{meerschaert2004finite, meerschaert2006finite}.

The Gr\"unwald approximation is completely identified by a generating function $W_1(z) = (1-z)^\alpha $ whose Taylor expansion provides coefficients for the series in the Gr\"unwald approximation formula.
This approximation is known to be of first order with the space
discretization size $h$ in the shifted and non-shifted form and
is, therefore, useful only in first order schemes for the FDE such as Euler (forward)
and implicit Euler(backward).
Since the CN approximation scheme is of second order in time step $\tau$,
Meerchaert et al. \cite{tadjeran2006second} used extrapolation improvement
for the space discretization to obtain a second order accuracy.
Subsequently, second order approximations for the space fractional
derivatives were obtained through some manipulations on the Gr\"unwald approximation.
Nasir et al. \cite{nasir2013second} obtained a second order accuracy through
a non-integer shift in the Gr\"unwald approximation, displaying super convergence.
A convex combination
of various shifts of the shifted Gr\"unwald approximations were used to obtain
higher order approximations in Chinese schools \cite{tian2015class,hao2015fourth, zhou2013quasi,li2017new,yu2017third},
some of them are unconditionally stable for the space FDE with CN type schemes. Zhao and Deng \cite{zhao2015series} extended the concept of super convergence to derive a series of  higher order approximations.
All of the above schemes are based on the first order approximating generating function $W_1(z) = (1-z)^\alpha $.

Earlier, Lubich \cite{lubich1986discretized} obtained some higher order approximations for the
fractional derivative without shift for orders up to 6.
Numerical experiments show that these approximations are also unstable for the
FDE when the methods mentioned above are used.
Shifted forms of these higher order approximations diminish the order to one,
making them unusable as Chen and Deng \cite{chen2014fourth, chen2014fourthSIAM} noted.
Nasir and Nafa \cite{nasir2018new, NasirNafa2018third} have recently constructed
generating functions for some higher order Gr\"unwald type approximations with shift.
They generalized the concept of Gr\"unwald difference approximation
by identifying it with coefficient generating functions. The generating functions have to be constructed individually for each order required using Taylor series expansions which are time consuming even with symbolic computations.

In this paper, we construct a general explicit form for approximating generating functions for any arbitrary order with shift. Their non-shift versions incidentally reduce to the Lubich generating functions. Our formulation also gives a unified general
formula for higher order approximations for any integer order derivatives.

The rest of the paper is organized as follows. In Section \ref{DefinitionSec}, some preliminaries, definitions and
notations of earlier works are introduced with proofs of relevant
important results. In Section \ref{ConstructSec},
the general explicit form for generating functions is established.
Section \ref{finitedif} derives from the main result finite difference formulas
for integer-order derivatives for any desired order and shift.
Conclusion are drawn in Section \ref{ConclusionSec}.

\section{Prelimineries, definitions and notations}\label{DefinitionSec}

The space of Lebesgue integrable functions defined in  $\Omega \subseteq \mathbb{R}$ is denoted by
$L_1(\Omega) $.
Let $f(x)\in L_1(\mathbb{R})$
and assume that it  is sufficiently differentiable so that the following definitions hold.
The left and right Riemann-Liouville fractional derivatives of real
order $\alpha >0 $ are defined respectively as
\begin{equation}\label{LeftRL}
  \;^{RL} D_{x-}^\alpha f(x) = \frac{1}{\Gamma (n-\alpha) } \frac{d^n}{dx^n}
  \int_{-\infty}^{x} \frac{f(\eta)}{(x-\eta)^{\alpha + 1 - n}} d\eta ,
\end{equation}
and
\begin{equation}\label{RightRL}
  \;^{RL} D_{x+}^\alpha f(x) = \frac{(-1)^n }{\Gamma (n-\alpha) } \frac{d^n}{dx^n}
  \int_{x}^{\infty} \frac{f(\eta)}{(\eta-x)^{\alpha + 1 - n}} d\eta ,
\end{equation}
where $n = [ \alpha] + 1  $, an integer with
$n-1 < \alpha \le n $ and $\Gamma (\cdot) $ denotes the gamma function.

The corresponding left/right Gr\"unwald-Letnikov (GL) definition  is given respectively by

\begin{equation}\label{LeftGL}
  \;^{GL} D_{x\mp}^\alpha f(x) = \lim_{h \rightarrow 0} \frac{1}{h^\alpha } \sum_{k=0}^\infty (-1)^k \binom{\alpha}{k} f(x \mp k h),
\end{equation}
where $\binom{\alpha}{k} = \frac{\Gamma (\alpha +1)}{\Gamma (\alpha +1-k)k!} $.

It is known that the Riemann-Liouville and Gr\"unwald-Letnikov definitions are
equivalent \cite{podlubny1998fractional}. Hence, from now onwards, we denote the left and right fractional differential operators
as $ D_{x-}^\alpha $ and $ D_{x+}^\alpha $ respectively

The Fourier transform (FT) of an integrable function $f(x) \in L_1(\mathbb{R})$ is defined by
$
\hat{f}(\eta) := \mathfrak{F}(f(x))(\eta) = \int_{-\infty}^{\infty} f(x)e^{-i\eta x}dx.$
The inverse FT is given by\\
$
\mathfrak{F}^{-1}({\hat f}(\eta))(x) = \int_{-\infty}^{\infty} {\hat f} (\eta)e^{i\eta x}d\eta = f(x).
$
The FT is linear:
$
\mathfrak{F}(\alpha f(x)+\beta g(x))(\eta) = \alpha \hat{f}(\eta)+ \beta {\hat g}(\eta).
$
For a function $f$ at a point $x+\beta \in \Omega,  \, \beta \in \mathbb{R} $, the FT is given by
$
\mathfrak{F}(f(x+\beta))(\eta) = e^{i\beta x} \hat{f}(\eta).
$
%
If $ D_{x\mp}^\alpha  f(x) \in L_1(\Omega) $,   the FT of the left and right fractional derivatives are given by
$
\mathfrak{F} (D_{x\mp}^\alpha  f(x))(\eta) =(\pm i\eta)^\alpha \hat{f}(\eta)
$ \cite{podlubny1998fractional}.
A consequence of Riemann-Lebesgue lemma states that $f(x)$ and its derivatives upto $\gamma>0$ are in $L_1(\mathbb{R})$ if and only if
$ |\hat{f}(\eta)| \le C_0(1+|\eta|)^{-\gamma} $ for some constant $C_0>0$.

For a fixed $h$, the Gr\"unwald approximations for the fractional derivatives in (\ref{LeftRL}) and (\ref{RightRL})
are obtained by simply dropping the limit in the GL definitions in (\ref{LeftGL}):
\begin{equation*}
  \delta_{x \mp,h}^\alpha f(x) = \frac{1}{h^\alpha } \sum_{k=0}^\infty g_k^{(\alpha)} f(x \mp kh),
\end{equation*}
where
$g_k^{(\alpha)} = (-1)^k \binom{\alpha}{k}$ are the coefficients of the
Taylor series expansion of the generating function $W_1(z) = (1-z)^\alpha $.

When $f(x) $ is defined in the intervals $[a,b]$, it is zero-extended outside the interval to adopt these definitions of fractional derivatives and their approximations. The
sums are then restricted to  finite terms up to $N$
which grows to infinity as $h \rightarrow 0$.
Ideally, $N$ is chosen to be $ N = \left[ \frac{x-a}{h} \right] $ and
$ N = \left[ \frac{b-x}{h} \right] $ for the left and right fractional derivatives  respectively to cover the sum up to the boundary of the domain intervals, where $[y]$ denotes the integer part of $y$.

The Gr\"unwald approximations are  of first order accuracy:
\begin{equation*}\label{FirstOrder}
    D_{x\mp}^\alpha f(x) = \delta_{x\mp,h}^\alpha f(x) +O(h)
\end{equation*}
and display unstable solutions
in the approximation of FDE by implicit and CN type schemes \cite{meerschaert2004finite}. As a remedy,
a shifted form of Gr\"unwald formula with  shift $r$ is used:
\begin{equation*}\label{LshiftedGApp}
    \delta_{x\mp, h,r}^\alpha f(x) = \frac{1}{h^\alpha } \sum_{k=0}^{N+r} g_k^{(\alpha)} f(x \mp (k-r)h),
\end{equation*}
where the upper limit of the summation has been adjusted to cover the shift $r$.

Meerchaert et al. \cite{meerschaert2004finite} showed that
for a shift $r = 1$ , the $\delta_{x\mp, h,1}^\alpha f(x)$
are also of first order approximations with unconditional stability restored
in implicit and CN type schemes for space FDEs.

For higher order approximations, Nasir et al. \cite{nasir2013second} derived a second order approximation by a non-integer shift, displaying super convergence:
\begin{equation*}\label{Nasir2nd}
\delta_{x \mp, h,\alpha/2} f(x) = D_{x\mp }^\alpha f(x) + O(h^2).
\end{equation*}

Chen et al.\cite{tian2015class} used convex combinations
of different shifted Gr\"unwald forms to obtain order 2 approximations:
\begin{equation*}\label{Tian2nd}
  \lambda_1 \delta_{x\mp, h,p} +\lambda_2 \delta_{x_\mp,h,q}
  = D_{x\mp}^\alpha f(x) + O(h^2)
\end{equation*}
with $\lambda_1 = \frac{\alpha - 2q}{2(p-q)} $ and
$\lambda_2 = \frac{\alpha - 2p}{2(p-q)} $  for
$(p,q) = (1,0) , (1,-1)$.

Hao et al. \cite{hao2015fourth} obtained an quasi-compact order 4 approximation.

Higher order approximations were also obtained earlier by Lubich
\cite{lubich1986discretized} establishing a connection with the
characteristic polynomials of multistep methods for ordinary differential
equations. Specifically,
if $ \rho(z), \sigma(z) $ are, the
characteristic polynomials of a multistep method of order of convergence $p$ \cite{henrici1962discrete},
then $\left(\frac{\sigma(1/z)}{\rho(1/z)}\right)^\alpha $ gives the coefficients for the
Gr\"unwald type approximation of same order for the fractional derivative of order $\alpha$.
From the backward multistep methods, Lubich \cite{lubich1986discretized} derived
higher order approximations of up to order six in the form
$ W_p (z) = (\sum_{j=1}^p \frac{1}{j} (1-z)^j )^\alpha $.
The generating functions
$W_p(z)$ given in Table \ref{LubichHighOrder} are of order $p$ for $ 1\le p \le 6 $.

\begin{table}[h]
  \centering
  \begin{tabular}{ll}
    \hline
 $W_1(z) = \left(1-z\right)^\alpha   $ \\
 $W_2(z) = \left(\frac{3}{2} -2z+\frac{1}{2}z^2\right)^\alpha $ \\
 $W_3(z) = \left(
\frac{11}{6} -3z+\frac{3}{2}z^2-\frac{1}{3}z^3
\right)^\alpha  $ \\
 $W_4(z) = \left(
\frac{25}{12} -4z+3z^2-\frac{4}{3}z^3 + \frac{1}{4} z^4 \right)^\alpha $ \\
 $W_5(z) = \left(
\frac{137}{60} -5z+5z^2-\frac{10}{3}z^3 + \frac{5}{4} z^4 -\frac{1}{5}z^5
\right)^\alpha $ \\
 $W_6(z) = \left(
\frac{49}{20} -6z+\frac{15}{2}z^2-\frac{20}{3}z^3 + \frac{15}{4} z^4 -\frac{6}{5}z^5 +\frac{1}{6} z^6
\right)^\alpha $ \\
    \hline
  \end{tabular}
  \caption{Lubich approximation generating functions}\label{LubichHighOrder}
\end{table}

These higher order approximations also suffer the same instability issues
without shift when used
in numerical schemes mentioned above. Shifted forms of these approximations
reduce the order to one, making them uninteresting\cite{nasir2018new}.

\subsection{Higher order approximations with shift}\label{SubSec}

A generalization of the Gr\"unwald approximation was proposed by Nasir and Nafa
\cite{nasir2018new, NasirNafa2018third} to construct higher order approximating generating functions that can be used
with shift without reducing the order of accuracy.
The definitions and results in their work are given below.

Let $ \left\{ w_{k,r}^{(\alpha)} \right\} $ be any real sequence and
\begin{equation*}
W(z) = \sum_{k=0}^{\infty} w_{k,r}^{(\alpha)} z^k
\end{equation*}
be its generating function.

For a sufficiently smooth function $f(x)$,
denote the left or right Gr\"unwald type operator  with shift $r$ and weights $w_{k,r}^{(\alpha)}$ as
\begin{equation}\label{Gr\"unwaldOp}
\Delta_{x\mp, h,r}^\alpha f(x) = \frac{1}{h^\alpha } \sum_{k=0}^\infty
w_{k,r}^{(\alpha)} f(x \mp (k-r) h).
\end{equation}

\begin{definition}\label{def1}
A sequence $ \left\{ w_{k,r}^{(\alpha)} \right\} $ ( or its generating function $W(z)$ ) is said to approximate the left or right fractional derivative
$D_{x\mp} ^\alpha  f(x) $ at $x$ with shift $r$ in the sense of Gr\"unwald if
\begin{equation}\nonumber\label{consistency1}
D_{x\mp}^\alpha  f(x) =
\lim_{h \rightarrow 0} \Delta_{x\mp, h,r}^\alpha f(x).
\end{equation}

\end{definition}

\begin{definition}\label{def2}
A sequence $ \left\{ w_{k,r}^{(\alpha)} \right\} $ ( or its generating function $W(z)$ ) is said
to approximate the fractional derivative $D_{x\mp}^\alpha  f(x) $
with shift $r$  and order $p \ge 1$ if
\begin{equation}\label{GrCondition}
D_{x\mp }^\alpha  f(x) =
\Delta_{x\mp, h,r}^\alpha f(x) + O(h^p).
\end{equation}
\end{definition}

The following theorem establishes an equivalent characterization
of the generator $W(z)$ for an approximation of fractional differential operator
with order  $ p\ge 1$ and shift $r$. We denote for conciseness the integer-order differential operators as $D^k = \frac{d^k}{dx^k}, k =0,1,2,\cdots $.

\begin{thm}{\rm \cite{nasir2018new, NasirNafa2018third}}\label{Mainthm}

Let $\alpha \in \mathbb{R}^+ $ and $m,n \in \mathbb{Z}^+ $ satisfying $n-1 < \alpha \le n < m$. Let $f(x)\in C^{m + n + 1}(\mathbb{R})$ and $ D^k f(x) \in L_1(\mathbb{R})$ for $ 0 \le k \le m+n+1$.  Let $W(z)$ be the generating function of a real sequence $\{ w_{k,r}^{(\alpha)} \}$ and
\begin{equation}\label{Grz}
G_r(z) := \frac{1}{z^\alpha} W(e^{-z})e^{rz}
\end{equation}
be analytic in $|z|\le R$ for some $ R\ge 1 $.
Then,
the generating function $W(z)$ approximates the left and right fractional differential
operators with  shift $r$ and order $p$, $1\le p\le m $, if and only if
\begin{equation}\label{Worderp}
   G_r(z) = 1 + O(z^p).
\end{equation}
Moreover, if $ G_r(z) = 1 + \sum_{k=p}^{\infty} a_{k} z^k $ , we have
\begin{equation}\label{TaylorFractional}
\Delta_{x\mp,h,r}^\alpha f(x) =
D_{x\mp}^\alpha f(x) +
 \sum_{k=p}^{m-1} h^{k} a_{k}  D_{x\mp}^{k+\alpha} f(x) + O(h^m).
\end{equation}
\end{thm}

\begin{proof}
We restrict the proof for the left fractional derivative and the proof for right derivative follows analogously. Analyticity of $G_r(z)$ allows to write $G_r(z) = \sum_{l=0}^\infty a_lz^l $ as a series.

Applying the FT for $\Delta_{x-,h,r}^\alpha f(x) $ in (\ref{Gr\"unwaldOp}),
with the use of linearity of FT, we have, with $z = i\eta h$,
\begin{align}
\mathfrak{F} (\Delta_{x-,h,r}^\alpha  f(x))(\eta)  & = \nonumber \frac{1}{h^\alpha } \sum_{k=0}^\infty
w_{k,r}^{(\alpha)} \mathfrak{F}(f(x-(k-r) h))(\eta)\\\nonumber
   & = \frac{1}{h^\alpha } \sum_{k=0}^\infty w_{k,r}^{(\alpha)}e^{-(k-r)i\eta h} \hat{f}(\eta)\\\nonumber
   & = \frac{e^{rih\eta}}{(i\eta h)^\alpha }  \sum_{k=0}^\infty w_{k,r}^{(\alpha)} e^{-ki\eta h}(i\eta)^\alpha \hat{f}(\eta)\\\nonumber
   & = \frac{e^{rz}}{z^\alpha}  \left[  \sum_{k=0}^\infty w_{k,r}^{(\alpha)} (e^{-z})^k  \right]
   (i\eta)^\alpha \hat{f}(\eta)
   = \frac{e^{rz}}{z^\alpha} W(e^{-z}) (i\eta)^\alpha \hat{f}(\eta) \\
   &= G_r(z) (i\eta)^\alpha \hat{f}(\eta)
    = \sum_{l=0}^{m-1} a_{l} (i \eta h)^{l+\alpha} {\hat f}(\eta) +
    \hat{\phi}(\eta, h), \label{Remainder}
   \end{align}
where, from (\ref{Remainder}),
\begin{equation}\label{PhiDef}
\hat{\phi}(\eta, h) = \left(G_r(i\eta h) -
\sum_{l=0}^{m-1} a_{l} (i \eta h)^{l} {\hat f}(\eta) \right) (i\eta)^\alpha \hat{f}(\eta).
\end{equation}
We establish, with the assumptions on $W(z)$ and $f(x)$, that
\begin{equation}\label{Phihatrelation}
|\hat{\phi}(\eta,h)| \le C h^m (1+|\eta|)^{-(n+1-\alpha)}
\end{equation}
so that, being $n+1-\alpha >0 $, there is a function $\phi(x,h) \in L_1(\mathbb{R}) $
such that $\mathfrak{F}(\phi(x,h)) = \hat{\phi}(\eta,h) $.
Taking the inverse FT, we have
\begin{equation}\label{MainApprox}
\Delta_{h,+r}^\alpha  f(x)
= \sum_{l=0}^{m-1} a_{l} D_{x-}^{l+\alpha}  f(x) h^l + O(h^m).
\end{equation}

Now, (\ref{Worderp}) holds if and only if $ a_{0} = 1, a_l = 0$, for $ l = 1,2,\cdots, p-1$
and the required
(\ref{TaylorFractional}) holds.

To prove (\ref{Phihatrelation}), we establish the following:
(a) We have from (\ref{Worderp}), $z^\alpha G_r(z) = W(e^{-z} ) e^{rz}, z \ne 0. $  Taking limit as $z\rightarrow 0$, we have $W(1) = 0$. This is the consistency condition for the approximation of the fractional operator.

(b) From the consistency condition, we have $W(z) = (1-z)^\alpha Q(z) $ for some analytic function $Q(z)$. Hence, we have
$|Q(e^{-i\eta h}) | \le C$. Moreover, the series of $G_r(z) $ is absolutely convergent with radius of convergence $R$.
Now, for $ |\eta h| > R $, we have
\begin{equation}\label{Estimate1}
|G_r(i\eta h)| =\left|\frac{1-e^{-i\eta h}}{i\eta h } \right|^\alpha |Q(e^{-i\eta h}) |\le C \frac{2^\alpha }{R^\alpha}
\le C \frac{2^\alpha |\eta h|^m }{R^{\alpha+m}}=C_1|\eta h|^m.
\end{equation}

(c) Also, for $ |\eta h| > R $, we have
\begin{equation}\label{Estimate2}
\left|\sum_{l=0}^{m-1} a_{l} (i \eta h)^{l} {\hat f}(\eta)\right|
\le
\sum_{l=0}^{m-1} |a_{l}| |\eta h|^{l}| {\hat f}(\eta)|
\le
|\eta h|^m \sum_{l=0}^{m-1} \frac{|a_{l}||{\hat f}(\eta)|}{|\eta h|^{m-l}}
\le  C_2 |\eta h|^m.
\end{equation}

From the inequalities in  (\ref{Estimate1}) and (\ref{Estimate2}), we have,
for $ |\eta h| > R $,
\begin{equation}\label{Bound1}
\left|
  G_r(i\eta h) - \sum_{l=0}^{m-1} a_{l} (i \eta h)^{l}
\right|
\le (C_1 +C_ 2)|\eta h|^m.
\end{equation}

(d) Moreover, for  $ |\eta h| \le R $, absolute convergence of $G_r(z)$ gives,
\begin{equation}\label{Bound2}
\left|
  G_r(i\eta h) - \sum_{l=0}^{m-1} a_{l} (i \eta h)^{l}
\right|
\le
|\eta h|^m \sum_{l=m}^{\infty} |a_{l}|
\le C_3 |\eta h|^m.
\end{equation}

Inequalities (\ref{Bound1}) and (\ref{Bound2}) give, with
$C_4 = \max ( C_1 +C_2 , C_3)$, for $ \eta h \in \mathbb{R}$,
\begin{equation}\label{Bound3}
\left|
  G_r(i\eta h) - \sum_{l=0}^{m-1} a_{l} (i \eta h)^{l}
\right|
\le C_4 |\eta h|^m.
\end{equation}

Now, from (\ref{PhiDef}), we establish (\ref{Phihatrelation}) as follows:
\begin{align*}
|\hat{\phi}(\eta,h)| & \le
C_4  |\eta|^{\alpha+m} |\hat{f}(\eta)|\le
C_4 (1+|\eta|)^{\alpha+m}C_0 (1+|\eta|)^{-(m+n+1)} \\
& \le C (1+|\eta|)^{-(n+1-\alpha)}.
\end{align*}
\end{proof}

We also need the following results on the determinant of Vandermonde matrix and its variant.

\begin{lem}
It is known that, for a finite sequence of parameters  $x_0, x_1, \cdots, x_q$, the determinant of the   Vandermonde matrix $V(x_0, x_1, \cdots, x_q)$
  of size $(q+1)$  is given by
  \[
  |V(x_0, x_1, \cdots, x_q)| = \left|
  \begin{array}{ccccc}
    1   & 1   & 1   & \cdots  & 1 \\
    x_0 & x_1 & x_2 & \cdots  & x_q \\
    x_0^2 & x_1^2 & x_2^2  & \cdots & x_q^2 \\
    \vdots & \vdots & \vdots  & \vdots & \vdots \\
    x_0^q & x_1^q & x_2^q &  \cdots & x_q^q
  \end{array} \right|
  = \prod_{0\le i < j \le p} (x_j - x_i).
  \]
\end{lem}

\begin{lem}
Let a variant of the Vandermonde matrix be
\[
U_2(x_0, x_1, \cdots, x_q) =
\left[
  \begin{array}{cccc}
    1   & 1     & \cdots &  1 \\
    x_0^2 & x_1^2 &   \cdots & x_q^2 \\
    x_0^3 & x_1^3 &   \cdots & x_q^3 \\
    \vdots & \vdots  & \vdots & \vdots \\
    x_0^{q+1} & x_1^{q+1} &   \cdots & x_q^{q+1}
  \end{array} \right].
\]

The determinant of $U_2$, which is of size $(q+1)$, is given by
  \begin{equation}\label{VVendermode}
  |U_2|
  = \prod_{0\le i < j \le q} (x_j - x_i)
  \left(  \sum_{m=0}^{q} \prod_{ \begin{smallmatrix}
                                               l=0 \\
                                               l \ne m
                                             \end{smallmatrix}}^q x_l \right).
  \end{equation}
\end{lem}

\begin{proof}
  It is clear that $|U_2|$ is a multi variable homogeneous function of total degree $ 2+3+\cdots+(q+1) = q(q+3)/2$. Moreover,
  for $i\ne j$, by replacing $x_i$ with $x_j$, we have $|U_2|=0$ and thus $(x_i - x_j) $ are factors of $|U_2|$, and there are $q(q+1)/2$ such factors as was in the Vandermonde determinant $|V|$. Hence,
  \[
  |U_2|= \prod_{0\le i < j \le q} (x_j - x_i) P(x_0, x_1, \cdots, x_q),
  \] where $P$ is a homogeneous polynomial of total degree
  $q(q+3)/2 -   q(q+1)/2 = q.$

  Since there are $q+1$ variables, $x_j, 0\le j\le q$, $P$ has the form, with constant $A$,
  \[
  P = A[(x_1x_2\cdots x_q) + (x_0x_2\cdots x_q) + \cdots + (x_1x_2\cdots x_{q-1})]
   = A   \left(  \sum_{m=0}^{q} \prod_{ \begin{smallmatrix}
                                               l=0 \\
                                               l \ne m
                                             \end{smallmatrix}}^q x_l \right).
   \]
   Equating one of the terms of $|U_2|$, we see that $A=1$.
\end{proof}

\section{Construction of higher order approximations}\label{ConstructSec}

We consider generating functions in the form
 $W(z) = (P(z))^\alpha$, where $P(z)$ is a polynomial.

In \cite{nasir2018new}, a second order approximating  generating functions  with shifts were constructed in the form
$W_{2,r}(z) = (\beta_0 +\beta_1 z +\beta_2 z^2)^\alpha$
with the use of
Theorem \ref{Mainthm}.
The construction process involved Taylor series expansion
of $G_r(z) $ and applying (\ref{GrCondition}) to form a set of equations and solving for the coefficients $\beta_j$.
As this involves many parameters such as $\alpha$, order $p$, shift
$r$, for higher order
approximations, this expansion process leads a tedious algebraic manipulations requiring symbolic computation with a considerable amount of computing time.

In this paper, we construct explicit expressions for the coefficients
$\beta_j$ for an approximating generating function of the form
\[
W_{p,r}(z) =  (\beta_0 +\beta_1 z +\cdots+ \beta_p z^p)^\alpha
\]
with shift $r$ and order $p$.

\begin{thm}
With the assumptions of Theorem \ref{Mainthm},
the generating function of the form $W_p(z) = (\sum_{j=0}^p \beta_j z^j)^\alpha $ approximates the left and right fractional derivatives
$D_{x\mp}f(x)$ at $x$ with shift $r$ and order $p \ge 1$ if and only if
the coefficients $\beta_j, 0\le j\le p$, govern the linear system
\begin{equation}\label{VanderMonde}
  \sum_{j=0}^p (\lambda - j)^n \beta_j = \delta_{1,n}, \qquad n = 0,1,\cdots, p,
\end{equation}
where $\lambda = r/\alpha$ and $\delta_{1,n}$ is the Kronecker delta having value of one for  $n=1$ and zero otherwise.
\end{thm}

\begin{proof}
  In view of Theorem \ref{Mainthm}, $W_{p,r}(z)$ approximates the fractional derivatives with shift $r$ and order $p$ if and only if
 , with $\lambda_j = \lambda - j$,
\begin{align*}
  G_r(z) & = \frac{1}{z^\alpha} W_p(e^{-z})e^{rz} =
  \frac{1}{z^\alpha} \left(\sum_{j=0}^p \beta_j e^{-jz}\right)^\alpha e^{rz} =
   \frac{1}{z^\alpha} \left(\sum_{j=0}^p \beta_j e^{(r/\alpha -j)z}\right)^\alpha \\
   & = \frac{1}{z^\alpha} \left(\sum_{j=0}^p \beta_j e^{\lambda_j z}\right)^\alpha =
   \left( \frac{1}{z} \sum_{j=0}^p \beta_j \sum_{n=0}^\infty \frac{1}{n!} \lambda_j^n z^n
      \right)^\alpha  = \left(\sum_{n=0}^\infty b_n z^{n-1}  \right)^\alpha \\
      & = \left( \frac{b_0}{z} + b_1 + \sum_{n=2}^\infty b_n z^{n-1}   \right)^\alpha = 1 + O(z^p),
\end{align*}
where
\[ b_n = \frac{1}{n!} \sum_{j=0}^p \lambda_j^n \beta_j z^{n-1},
\qquad n = 0,1,2,\cdots . \]

Since $G_r(z)$ is analytic, it does not have any poles and hence we have $b_0 = 0$. Next, since $G_r(z) = 1 + O(z^p) $ has the constant term 1, we must have $ b_1 = 1$. \\
For $p=1$, these two conditions give the system
(\ref{VanderMonde}) and the proof ends.\\
For $p>1$, $G_r(z)$ reduces to
\[
G_r(z) = \left(
1+ \sum_{n=2}^\infty b_n z^{n-1} \right)^\alpha =: (1+X)^\alpha,
\]
where $X =  \sum_{n=2}^\infty b_n z^{n-1} $.
Taylor expansion of $G_r(z) $ with respect to $X$ gives
\[
G_r(z) = 1+\alpha X + \frac{\alpha(\alpha-1)}{2!} X^2
+ \frac{\alpha(\alpha-1)(\alpha-2)}{3!} X^3+\cdots = 1 +O(z^p).
\]

The term $z$ appears in the term $\alpha X$ only on the left hand side.
This gives $b_2 = 0$. The same is true for all $b_n, n = 2,3,\cdots, p$
by successively comparing the coefficients of $z^{n-1}, n = 2, 3, \cdots, p$.
Altogether, we have $b_n = \delta_{1,n} , n = 0,1,2,\cdots, p$ which yield the system (\ref{VanderMonde}).
\end{proof}

The system (\ref{VanderMonde}) can be expressed in Vandermonde matrix form as
\begin{equation}\label{VanderMatrix}
  V(\lambda_0, \lambda_1,\cdots, \lambda_p) {\bf b } = {\bf d},
\end{equation}
where ${\bf b,d} $ are column vectors given by
$ {\bf b} = [\beta_0,\beta_1,\cdots, \beta_p]^T , {\bf d} = [ 0,1,0,\cdots,0]^T$ respectively with superscript $T$ denoting transpose.
Solving the Vandermonde system (\ref{VanderMatrix}), we have the following main result.

\begin{thm}\label{BetaSolution}
  The generating function $W_{p,r}(z) = (\beta_0 + \beta_1 z +\cdots + \beta_p z^p)^\alpha $ approximates the fractional derivatives with shift $r$ and order $p$ if and only if the coefficients $\beta_j$  are given by
\begin{align}\label{ExplicitBeta1}
\beta_j & = - \left(
\prod_{\begin{smallmatrix} m=0\\ m\ne j \end{smallmatrix}   }^{p} \frac{1}{j-m}
\right) \left(
\sum_{\begin{smallmatrix} m=0\\ m\ne j \end{smallmatrix}}^p \prod_{\begin{smallmatrix} l=0\\ l\ne m,j \end{smallmatrix}}^{p} (\lambda-l)
\right) \\
& = -\left( \prod_{\begin{smallmatrix} m=0\\ m\ne j \end{smallmatrix}}^p \frac{\lambda - m}{j-m}  \right) \left(
    \sum_{\begin{smallmatrix} m=0\\ m\ne j \end{smallmatrix}}^p \frac{1}{\lambda - m}
    \right), \qquad j = 0,1,2,\cdots, p. \label{ExplicitBeta2}
  \end{align}

\end{thm}

\begin{proof}
  Employing Cramer's rule to solve the system (\ref{VanderMatrix}), we have
\begin{equation}\label{Cramersol}
  \beta_j = \frac{|V_j({\bf d})| }{|V(\lambda_0, \lambda_1,\cdots, \lambda_p) |},
\end{equation}
where $V_j({\bf d})$ is the matrix obtained from $V(\lambda_0, \lambda_1,\cdots, \lambda_p)$ by replacing its $j^{\rm th}$ column with
  vector ${\bf b}$.\\
When $p = 1$, we have $|V_0({\bf d} )| =
\left| \begin{array}{cc}
         0 & 1 \\
         1 & \lambda_1
       \end{array}\right| = -1$
 and $|V_1({\bf d} )| =
\left| \begin{array}{cc}
         1 & 0 \\
         \lambda_0 & 1
       \end{array}\right| = 1$,
while $|V(\lambda_0,\lambda_1)| =
\lambda_1-\lambda_0 = 0-1=-1$. Hence, $\beta_0 = 1, \beta_1 = -1$.

When $p >1$, $|V_j(\bf d)|$ is given by, with $\bf d$ in the $j$-th column and by using the standard determinant rule selecting the $j-$th column as pivot,
\begin{align*}
|V_j({\bf d})|  & = \left|
\begin{array}{cccccc}
1 & 1 & \cdots & 0 & \cdots & 1 \\
\lambda_0 & \lambda_1 & \cdots & 1 & \cdots & \lambda_p \\
\lambda_0^2 & \lambda_1^2 & \cdots & 0 & \cdots & \lambda_p^2 \\
\vdots  & \vdots & \vdots & \vdots & \vdots & \vdots \\
\lambda_0^p & \lambda_1^p & \cdots & 0 & \cdots & \lambda_p^p
\end{array}
\right| \\
 & = (-1)^{j+1}
\left|
\begin{array}{ccccccc}
1 & 1 & \cdots & 1 &  1 & \cdots & 1 \\
\lambda_0^2 & \lambda_1^2 & \cdots & \lambda_{j-1}^2 &\lambda_{j+1}^2 &  \cdots & \lambda_p^2 \\
\lambda_0^3 & \lambda_1^3 & \cdots & \lambda_{j-1}^3 &\lambda_{j+1}^3 &  \cdots & \lambda_p^3 \\
\vdots  & \vdots & \vdots & \vdots & \vdots & \vdots  & \vdots \\
\lambda_0^p & \lambda_1^p & \cdots & \lambda_{j-1}^p &\lambda_{j+1}^p &  \cdots & \lambda_p^p
\end{array}
\right| .
\end{align*}

Note that the last determinant is of the form $U_2$ of size $p\times p$ missing the column with the parameter $\lambda_j$. Hence, we have from (\ref{VVendermode}),
\[
|V_j({\bf d)}|   = (-1)^{j+1} \left(\prod_{\begin{smallmatrix} 0\le m<n\le p\\ m,n\ne j \end{smallmatrix}} (\lambda_n-\lambda_m) \right)
\left(
\sum_{\begin{smallmatrix} m=0\\ m\ne j \end{smallmatrix}}^{p} \prod_{\begin{smallmatrix} l=0\\ l\ne m,j \end{smallmatrix}}^{p} \lambda_l
\right).
\]
The determinant of the Vandermonde matrix $V(\lambda_0, \lambda_1,\cdots, \lambda_p))$ can be written as
\[
|V|
= \prod_{ 0\le m<n\le p }
 (\lambda_n-\lambda_m) = \prod_{\begin{smallmatrix} 0\le m<n\le p\\ m,n\ne j \end{smallmatrix}} (\lambda_n-\lambda_m)
 (-1)^j \prod_{\begin{smallmatrix} m=0\\ m\ne j \end{smallmatrix} }^p (\lambda_m-\lambda_j),
\]
where we have partitioned the product to extract the product part in $|V_j({\bf d})|$.

We thus have from (\ref{Cramersol}) and with the observation that $\lambda_m - \lambda_j = j - m$,
\[
\beta_j = - \left(
\prod_{\begin{smallmatrix} m=0\\ m\ne j \end{smallmatrix}   }^{p} \frac{1}{j-m}
\right) \left(
\sum_{\begin{smallmatrix} m=0\\ m\ne j \end{smallmatrix}}^p \prod_{\begin{smallmatrix} l=0\\ l\ne m,j \end{smallmatrix}}^{p} \lambda_l
\right)=
-\left( \prod_{\begin{smallmatrix} m=0\\ m\ne j \end{smallmatrix}}^p \frac{\lambda - m}{j-m}  \right) \left(
    \sum_{ \begin{smallmatrix} m=0\\ m\ne j \end{smallmatrix}}^p \frac{1}{\lambda - m}
    \right).
\]
\end{proof}

Note that formulas (\ref{ExplicitBeta1}) and (\ref{ExplicitBeta2}) are compatible for the case $p=1$ as well although we treated it separately in the proof.

We derive some generating functions using the explicit forms (\ref{ExplicitBeta1}) and (\ref{ExplicitBeta2}).
\begin{enumerate}
\item When $p =1 $, we have from (\ref{ExplicitBeta2}) \\
$\beta_0 = -\frac{\lambda_1}{0-1} \frac{1}{\lambda_1} = 1 ,
\beta_1 = -\frac{\lambda_0}{1-0} \frac{1}{\lambda_0} = -1$ and generating function is given by $W_{1,r}(z) = (1-z)^\alpha$.

\item When $p =2 $, we have \\
 $\beta_0 = -\frac{\lambda_1}{0-1}\frac{\lambda_2}{0-2} \left(\frac{1}{\lambda_1} + \frac{1}{\lambda_2} \right)= -\frac{1}{2}(\lambda_2+\lambda_1) = -\frac{1}{2} (\lambda-2+\lambda-1) = \frac{3}{2} -\lambda $  ,\\
 $\beta_1 = -\frac{\lambda_0}{1-0}\frac{\lambda_2}{1-2} \left(\frac{1}{\lambda_0} + \frac{1}{\lambda_2} \right) = -1(\lambda_2+\lambda_0) = -\frac{1}{2}(\lambda-2+\lambda-0) = -2 + 2\lambda $  ,\\
 $\beta_2 = -\frac{\lambda_0}{2-0}\frac{\lambda_1}{2-1} \left(\frac{1}{\lambda_0} + \frac{1}{\lambda_1} \right)= -\frac{1}{2}(\lambda_1+\lambda_0) = -\frac{1}{2}(\lambda-1+\lambda-0) = \frac{1}{2} -\lambda $  \\
 and hence
 $W_{2,r}(z) = \left( \left(  \frac{3}{2} - \lambda \right) + (-2 + 2\lambda) z + \left(\frac{1}{2} -\lambda \right) z^2 \right)^\alpha $ as in \cite{nasir2018new}.
\item When $p = 3$, using the other form for $\beta_j$, we have \\
 $\beta_0 = -\frac{1}{0-1}\frac{1}{0-2}\frac{1}{0-3} (\lambda_2\lambda_3  + \lambda_1\lambda_3  + \lambda_1\lambda_2 ) = \frac{1}{2}\lambda^2 -2\lambda +\frac{11}{6}$.\\
 $\beta_1 = -\frac{1}{1-0}\frac{1}{1-2}\frac{1}{1-3} (\lambda_2\lambda_3  + \lambda_0\lambda_3  + \lambda_0\lambda_2 ) = -\frac{3}{2}\lambda^2 + 5\lambda - 3$.\\
 $\beta_2 = -\frac{1}{2-0}\frac{1}{2-1}\frac{1}{2-3} (\lambda_1\lambda_3  + \lambda_0\lambda_3  + \lambda_0\lambda_1 ) = -\frac{3}{2}\lambda^2 - 4\lambda + \frac{3}{2}$.\\
 $\beta_3 = -\frac{1}{3-0}\frac{1}{3-1}\frac{1}{3-2} (\lambda_1\lambda_2  + \lambda_0\lambda_2  + \lambda_1\lambda_2 ) = -\frac{1}{2}\lambda^2 + \lambda - \frac{1}{3}$.
\end{enumerate}

We list the coefficients for approximating generating functions $W_{p,r}(z)$ with order $p$, shift $r$ obtained for $1\le p\le 6$
 in Table \ref{Wtable}.

\begin{table}[th]
\centering
\begin{tabular}{ll}
\hline\hline
$p$ & $\beta_k, 0\le k \le p$\\
\hline\hline
1 & $ \beta_{0} = 1 $ \\
& $ \beta_{1} = -1 $ \\
\hline
2 & $ \beta_{0} = - \lambda + \frac{3}{2} $ \\
& $ \beta_{1} = - 2 +2 \lambda $ \\
& $ \beta_{2} = - \lambda + \frac{1}{2} $ \\
\hline
3 & $ \beta_{0} = \frac{\lambda^{2}}{2} - 2 \lambda + \frac{11}{6} $ \\
& $ \beta_{1} = - \frac{3 \lambda^{2}}{2} + 5 \lambda - 3 $ \\
& $ \beta_{2} = \frac{3 \lambda^{2}}{2} - 4 \lambda + \frac{3}{2} $ \\
& $ \beta_{3} = - \frac{\lambda^{2}}{2} + \lambda - \frac{1}{3} $ \\
\hline
4 & $ \beta_{0} = - \frac{\lambda^{3}}{6} + \frac{5 \lambda^{2}}{4} - \frac{35 \lambda}{12} + \frac{25}{12} $ \\
& $ \beta_{1} = \frac{2 \lambda^{3}}{3} - \frac{9 \lambda^{2}}{2} + \frac{26 \lambda}{3} - 4 $ \\
& $ \beta_{2} = - \lambda^{3} + 6 \lambda^{2} - \frac{19 \lambda}{2} + 3 $ \\
& $ \beta_{3} = \frac{2 \lambda^{3}}{3} - \frac{7 \lambda^{2}}{2} + \frac{14 \lambda}{3} - \frac{4}{3} $ \\
& $ \beta_{4} = - \frac{\lambda^{3}}{6} + \frac{3 \lambda^{2}}{4} - \frac{11 \lambda}{12} + \frac{1}{4} $ \\
\hline
5 & $ \beta_{0} = \frac{\lambda^{4}}{24} - \frac{\lambda^{3}}{2} + \frac{17 \lambda^{2}}{8} - \frac{15 \lambda}{4} + \frac{137}{60} $ \\
& $ \beta_{1} = - \frac{5 \lambda^{4}}{24} + \frac{7 \lambda^{3}}{3} - \frac{71 \lambda^{2}}{8} + \frac{77 \lambda}{6} - 5 $ \\
& $ \beta_{2} = \frac{5 \lambda^{4}}{12} - \frac{13 \lambda^{3}}{3} + \frac{59 \lambda^{2}}{4} - \frac{107 \lambda}{6} + 5 $ \\
& $ \beta_{3} = - \frac{5 \lambda^{4}}{12} + 4 \lambda^{3} - \frac{49 \lambda^{2}}{4} + 13 \lambda - \frac{10}{3} $ \\
& $ \beta_{4} = \frac{5 \lambda^{4}}{24} - \frac{11 \lambda^{3}}{6} + \frac{41 \lambda^{2}}{8} - \frac{61 \lambda}{12} + \frac{5}{4} $ \\
& $ \beta_{5} = - \frac{\lambda^{4}}{24} + \frac{\lambda^{3}}{3} - \frac{7 \lambda^{2}}{8} + \frac{5 \lambda}{6} - \frac{1}{5} $ \\
\hline
6 & $ \beta_{0} = - \frac{\lambda^{5}}{120} + \frac{7 \lambda^{4}}{48} - \frac{35 \lambda^{3}}{36} + \frac{49 \lambda^{2}}{16} - \frac{203 \lambda}{45} + \frac{49}{20} $ \\
& $ \beta_{1} = \frac{\lambda^{5}}{20} - \frac{5 \lambda^{4}}{6} + \frac{31 \lambda^{3}}{6} - \frac{29 \lambda^{2}}{2} + \frac{87 \lambda}{5} - 6 $ \\
& $ \beta_{2} = - \frac{\lambda^{5}}{8} + \frac{95 \lambda^{4}}{48} - \frac{137 \lambda^{3}}{12} + \frac{461 \lambda^{2}}{16} - \frac{117 \lambda}{4} + \frac{15}{2} $ \\
& $ \beta_{3} = \frac{\lambda^{5}}{6} - \frac{5 \lambda^{4}}{2} + \frac{121 \lambda^{3}}{9} - 31 \lambda^{2} + \frac{254 \lambda}{9} - \frac{20}{3} $ \\
& $ \beta_{4} = - \frac{\lambda^{5}}{8} + \frac{85 \lambda^{4}}{48} - \frac{107 \lambda^{3}}{12} + \frac{307 \lambda^{2}}{16} - \frac{33 \lambda}{2} + \frac{15}{4} $ \\
& $ \beta_{5} = \frac{\lambda^{5}}{20} - \frac{2 \lambda^{4}}{3} + \frac{19 \lambda^{3}}{6} - \frac{13 \lambda^{2}}{2} + \frac{27 \lambda}{5} - \frac{6}{5} $ \\
& $ \beta_{6} = - \frac{\lambda^{5}}{120} + \frac{5 \lambda^{4}}{48} - \frac{17 \lambda^{3}}{36} + \frac{15 \lambda^{2}}{16} - \frac{137 \lambda}{180} + \frac{1}{6} $ \\
\hline
\end{tabular}
\caption{Coefficients $\beta_k$ of
$W_{p,r}(z) = (\beta_0 + \beta_1 z + \cdots + \beta_p z^p  )^\alpha $} for order $p$ approximation, $1\le p\le 6$.\label{Wtable}
\end{table}

Note that when there is no shift, i.e., $r = 0$, we have $\lambda = \frac{r}{\alpha} = 0$ and  these generating functions reduce to those in Table \ref{LubichHighOrder} obtained by Lubich \cite{lubich1986discretized}.

When constructing these generating functions, we were not concerned with the condition
of absolute convergence of the series for the generating function. Therefore, these results should be validated
using this condition. This validation will give range of values for $\alpha$ and for the shift $r$.

For the generating functions $W_{p,r}(z) $, we have the following:

\begin{thm}
The generating functions $W_{p,r}(z) $ approximate the left and right
fractional differential operators $D_{x\mp}^{\alpha} $  with shift $r$ and order $p$.
\end{thm}
\begin{proof}
  Taylor series expansion of $G_{p,r}(z)$ gives  $ G_{p,r}(z) = \frac{1}{z^\alpha}e^{\pm rz}
  W_{p,r}(e^{-z}) = 1+O(z^p)$ for $ 0 \le p \le 6$.
\end{proof}

The Gr\"unwald weights $w_{k,r}^{(\alpha)} $ can be computed by a recurrence formula
\cite{rall1981automatic, weilbeer2005efficient}.

\begin{lem}
Let $P(z) = \sum_{j=0}^{\infty} \beta_j z^j $, where $\beta_j \in \mathbb{R}$ , $\alpha >0$
and $W(z) = ( P(z) )^\alpha = \sum_{m=0}^{\infty} w_m z^m $. Then,
the coefficients $w_m$ satisfy the recurrence form
\begin{equation}\label{Miller}
w_{0} = \beta_0^\alpha , \;
w_{m} = \frac{1}{m \beta_0}
\sum_{j=1}^{m} (j(\alpha+1)-m) w_{m-j} \beta_j.
\end{equation}
\end{lem}

For the generating function $W_{p,r}(z)$, being the power of a polynomial of degree $p$,
the upper limit of the summation in (\ref{Miller}) goes up to
$\min(m,p) $ only.

It seems that, despite
 higher order of approximation of $W_{p,r}(z)$, the stability of solutions
using these approximations to FDEs remains an issue. Our numerical tests
on some steady state problems show that the second order approximation
$W_{2,1}(z) $ displays stability for values of $\alpha $ in the interval $[1,2]$.
However, for the other  higher order approximations,
the stability is limited to a subset $[\alpha_0, 2], 1<\alpha_0, $   of the
interval $[1,2]$. This phenomenon warrants a thorough theoretical investigation
for approximations of orders three and above.

\section{General finite difference formula for derivatives}\label{finitedif}

The generating function (\ref{ExplicitBeta1}) also gives
a unified
finite difference formula for integer order derivatives.

Infact, when $\alpha = n $, an integer, the generating function $W_{p,r}(z)$ is an integer power of a polynomial and hence has a finite series expansion which  obviously satisfies the condition of absolute convergence.

This means that we have  finite difference formulas for any given order
with any given shift for any integer order derivatives.
The coefficients for the finite difference formulas are obtained by expanding the integer power of the polynomial
obtained from Theorem \ref{BetaSolution}.

We list here some of well known and not well-known finite difference formulas for demonstration. Let $f$ be a function defined
on an interval which has enough derivatives to have
finite difference forms. For notational brevity, we denote
$f(x + k h) = f_k $.\\
  1.
  For the first derivative $df(x)/dx$,  approximation of order 2 with shift 1 is obtained from (\ref{ExplicitBeta1}) with parameters, $\alpha = 1, p = 2$ and $ r = 1$. The generating function is then given by, after expanding the power,
  $
W(z) =  - \frac{z^{2}}{2} + \frac{1}{2}
$
    and we have
$\frac{d^{  }}{dx^{  }}f(x) =
\frac{1}{h}
\left(
\frac{1}{2} f_ { -1 }
- \frac{1}{2} f_ { 1 }
\right) +O(h^2).$
This is the well known central difference formula
of order 2 for the first derivative.\\

For order 3 with shift 2, we have
$
W(z) =  - \frac{z^{3}}{3} - \frac{z^{2}}{2} + z - \frac{1}{6}
$
    and the finite difference formula at $x$,
    with $\alpha = 1, p = 3 $ and $  r = 2$, is
$
\frac{d^{  }}{dx^{  }}f(x) =
\frac{1}{h}  \left(- \frac{1}{6} f_{ -2 }  +
 f_ {-1 }
- \frac{1}{2} f_{ 0 }
- \frac{1}{3} f_{ 1 }\right) + O(h^3).
$ \\

For order 3 with shift 3/2, we have
$
W(z) =  \frac{z^{3}}{24} - \frac{9 z^{2}}{8} + \frac{9 z}{8} - \frac{1}{24}
$
    and the finite difference formula at $f(x) = f_0$,
    with $ \alpha = 1, p =  3 $ and $ r =  3/2 $,  is given by
$
\frac{d^{  }}{dx^{  }}f(x) =
- \frac{1}{24} f_ { -3/2 }  +
\frac{9}{8} f_ { -1/2 }
- \frac{9}{8} f_ { 1/2 }  +
\frac{1}{24} f_ { 3/2 } +O(h^3)
$
which is a central difference formula with functional values at midpoints of the discretized points.

2.
For the second derivative $d^2f(x)/dx^2$,
approximating generating function of order 4 with shift 2 is given by
$
W_{4,2}(z) =  \left(- \frac{z^{4}}{12} + \frac{z^{3}}{2} - \frac{3 z^{2}}{2} + \frac{5 z}{6} + \frac{1}{4}\right)^{2}
 =  \frac{z^{8}}{144} - \frac{z^{7}}{12} + \frac{z^{6}}{2} - \frac{59 z^{5}}{36} + \frac{73 z^{4}}{24} - \frac{9 z^{3}}{4} - \frac{z^{2}}{18} + \frac{5 z}{12} + \frac{1}{16}
$
    and the finite difference formula
    for $ d^2f(x)/dx^2$
    at $x$ is given by
$
\frac{d^{ 2 }}{dx^{ 2 }}f(x) =
\frac{1}{h^2}  \left(
\frac{f_ { -2 }}{16}   +
\frac{5f_ { -1 }}{12}
- \frac{f_ { 0 }}{18}
- \frac{9f_ { 1 }}{4}   +
\frac{73f_ { 2 }}{24}
- \frac{59f_ { 3 }}{36}   +
\frac{f_ { 4 }}{2}
- \frac{f_ { 5 }}{12}   +
\frac{f_ { 6 }}{144} \right) + O(h^4)
$.\\

To the best of our knowledge, this type of general explicit
form for finite difference formulas with arbitrary shifts and order for any derivative
is not available in the literature and hence  is claimed novelty.
However, it should be mentioned that these finite difference formulas
with given stencils of points and gien order of derivatives can be computed by solving their corresponding system of equations individually for each case \cite{CameronTaylor}. A recursive algorithm to compute
the coefficients are also given in \cite{fornberg1988generation} in which the coefficients for the approximation of a higher order derivative require the
coefficients of the previous order derivative and the stencil points are supplied.\\

In our explicit form, the coefficients are directly
computed for a derivative of given order $n$ and approximation order $p$. The number of stencil points is then at most $ np+1 $. The derivative point within the stencil is given through the shift $r$.
The generating function for left/right derivatives of integer orders with no shift gives the backward/forward difference formulas and
the left or right derivative with halfway shift, that is
$ r = np/2 $ which may not be an integer,  gives the central difference formulas for the desired orders. Any other shifts give new formulas of intended orders.

\section{Conclusion}\label{ConclusionSec}

Higher order approximations for fractional derivatives in the Gr\"unwald sense are considered. The Gr\"unwald type approximations are completely characterised by corresponding generating functions from which the weights  are obtained. Construction of approximating generating functions in the form of power of polynomials is considered and
a general explicit form for such generating functions are
derived with arbitrary shift and order of accuracy for fractional derivatives. Incidentally, these generating functions are
generalizations of the
known higher order approximations without shift.

The formula for approximating generating function
also gives finite difference formulas for any integer-order derivatives as well giving new formulas with
any order and shift.

These explicit generating functions might be a useful
tool for constructing finite difference approximations for fractional
and integer-order
derivatives and for their analysis and computations.


\end{document}